\documentclass[12pt]{article}
\usepackage{amsmath,latexsym,amssymb}
\usepackage{enumerate}
\usepackage{amsthm}
\usepackage{epsfig,graphicx}

\newcommand\NN{\mathbb{N}}
\newcommand\RR{\mathbb{R}}
\newcommand\ZZ{\mathbb{Z}}

\newcommand\BB{\mathcal{B}}

\newcommand\GG{\mathcal{G}}
\newcommand\HH{\mathcal{H}}

\newcommand\TT{\mathcal{T}}
\newcommand\eps{{\varepsilon}}
\DeclareMathOperator\dist{dist}

\DeclareMathOperator\diam{diam}
\DeclareMathOperator\defin{def}
\DeclareMathOperator\height{height}
\renewcommand{\mod}{\operatorname{mod}}

\newtheorem{theorem}{Theorem}[section]
\newtheorem{definition}[theorem]{Definition}
\newtheorem{corollary}[theorem]{Corollary}

\newtheorem{lemma}[theorem]{Lemma}

\newtheorem{proposition}[theorem]{Proposition}

\newcommand{\address}{Address: Department of Mathematics, University of North Texas, 1155 Union Circle \#311430, Denton, TX 76203-5017, USA; E-mail: allaart@unt.edu}

\title{Level sets of signed Takagi functions}
\author{Pieter C. Allaart \footnote{\address}}

\begin{document}

\maketitle

\begin{abstract}

This paper examines level sets of functions of the form
\begin{equation*}
f(x)=\sum_{n=0}^\infty \frac{r_n}{2^n}\phi(2^n x),
\end{equation*}
where $\phi(x)=\dist(x,\ZZ)$, the distance from $x$ to the nearest integer, and $r_n=\pm\,1$ for each $n$. Such functions are referred to as {\em signed Takagi functions}. The case when $r_n=1$ for all $n$ is the classical Takagi function, a well-known example of a continuous but nowhere differentiable function. For $f$ of the above form, the maximum and minimum values of $f$ are expressed in terms of the sequence $\{r_n\}$. It is then shown that almost all level sets of $f$ are finite (with respect to Lebesgue measure on the range of $f$), but the set of ordinates $y$ with an uncountable level set is residual in the range of $f$. The concept of a local level set of the Takagi function, due to Lagarias and Maddock, is extended to arbitrary signed Takagi functions. It is shown that the average number of local level sets contained in a level set of $f$ is the reciprocal of the height of the graph of $f$, and consequently, this average lies between $3/2$ and $2$. These results generalize recent findings by Buczolich [{\em Acta Math. Hungar.} {121} (2008), 371--393], Lagarias and Maddock [{\em Monatsh. Math.} {166} (2012), 201--238], and Allaart [{\em Monatsh. Math.} {167} (2012), 311-331].


\end{abstract}

\section{Introduction}

Takagi's continuous nowhere differentiable function is defined by
\begin{equation}
T(x):=\sum_{n=0}^\infty \frac{1}{2^n}\phi(2^n x),
\label{eq:Takagi-def}
\end{equation}
where $\phi(x)=\dist(x,\ZZ)$, the distance from $x$ to the nearest integer. It was introduced in 1903 by Takagi \cite{Takagi} and rediscovered later by Van der Waerden \cite{vdW} and Hildebrandt \cite{Hildebrandt}, among others.
Much has been written about this function, especially in the last few decades. The reader is referred to recent surveys by Allaart and Kawamura \cite{AK} and Lagarias \cite{Lagarias} for various properties and applications of the Takagi function. Recent research has focused mainly on the level sets $L_T(y):=\{x\in[0,1]:T(x)=y\}$. Many authors (e.g. Kahane \cite{Kahane}, Baba \cite{Baba}) have pointed out that the maximum value of $T$ is $2/3$, and $L_T(2/3)$ is a Cantor set of Hausdorff dimension $1/2$. Recently, De Amo et al. \cite{ABDF} proved that $1/2$ is in fact the largest dimension of any level set of $T$, improving on a result of Maddock \cite{Maddock}. The cardinalities of the level sets of $T$ have also been investigated. Buczolich \cite{Buczolich} showed that $L_T(y)$ is finite for Lebesgue-almost every $y$. The present author studied the level sets further and proved, among other things, that every even positive integer occurs as the cardinality of some level set \cite{Allaart2}, and that the set of ordinates $y$ such that $L_T(y)$ is infinite is residual in the range of $T$ \cite{Allaart3}. Another result along these lines, due to Lagarias and Maddock \cite{LagMad2}, is that the set of ordinates $y$ for which $L_T(y)$ has strictly positive Hausdorff dimension is of full Hausdorff dimension 1.

Lagarias and Maddock \cite{LagMad1,LagMad2} also introduced the interesting concept of a {\em local level set} of the Takagi function. Local level sets are subsets of (global) level sets whose members are obtained from one another by simple combinatorial operations on their binary expansions. Local level sets have a particularly simple structure: they are either finite or Cantor sets. In \cite{LagMad1}, it is shown that the average number of local level sets contained in a level set of $T$ is $3/2$. On the other hand, there is a residual set of ordinates $y$ for which $L_T(y)$ contains infinitely many local level sets, and an uncountable dense set of ordinates $y$ for which $L_T(y)$ contains {\em uncountably} many local level sets; see \cite[Theorem 5.2]{Allaart3}.

This paper generalizes some of the above-mentioned results to an infinite class of functions obtained by multiplying the summands in \eqref{eq:Takagi-def} by arbitrary signs. Specifically, let ${\bf r}=(r_0,r_1,\dots)\in\{-1,1\}^{\NN}$, and define
\begin{equation}
f(x):=f_{\bf r}(x):=\sum_{n=0}^\infty \frac{r_n}{2^n}\phi(2^n x).
\label{eq:general-function}
\end{equation}
We denote the collection of all such functions $f$ by $\TT_{\pm}$, and call them {\em signed Takagi functions}. Two examples are shown in Figure \ref{fig:general-examples}. Like the Takagi function, each member of $\TT_{\pm}$ is clearly $1$-periodic and continuous. That it is also nowhere differentiable follows from Theorem 2 of K\^ono \cite{Kono} or by a straightforward modification of Billingsley's argument \cite{Billingsley}. Let 
$$L_f(y):=\{x\in[0,1]:f(x)=y\}, \qquad y\in\RR.$$

\begin{figure}
\begin{center}
\epsfig{file=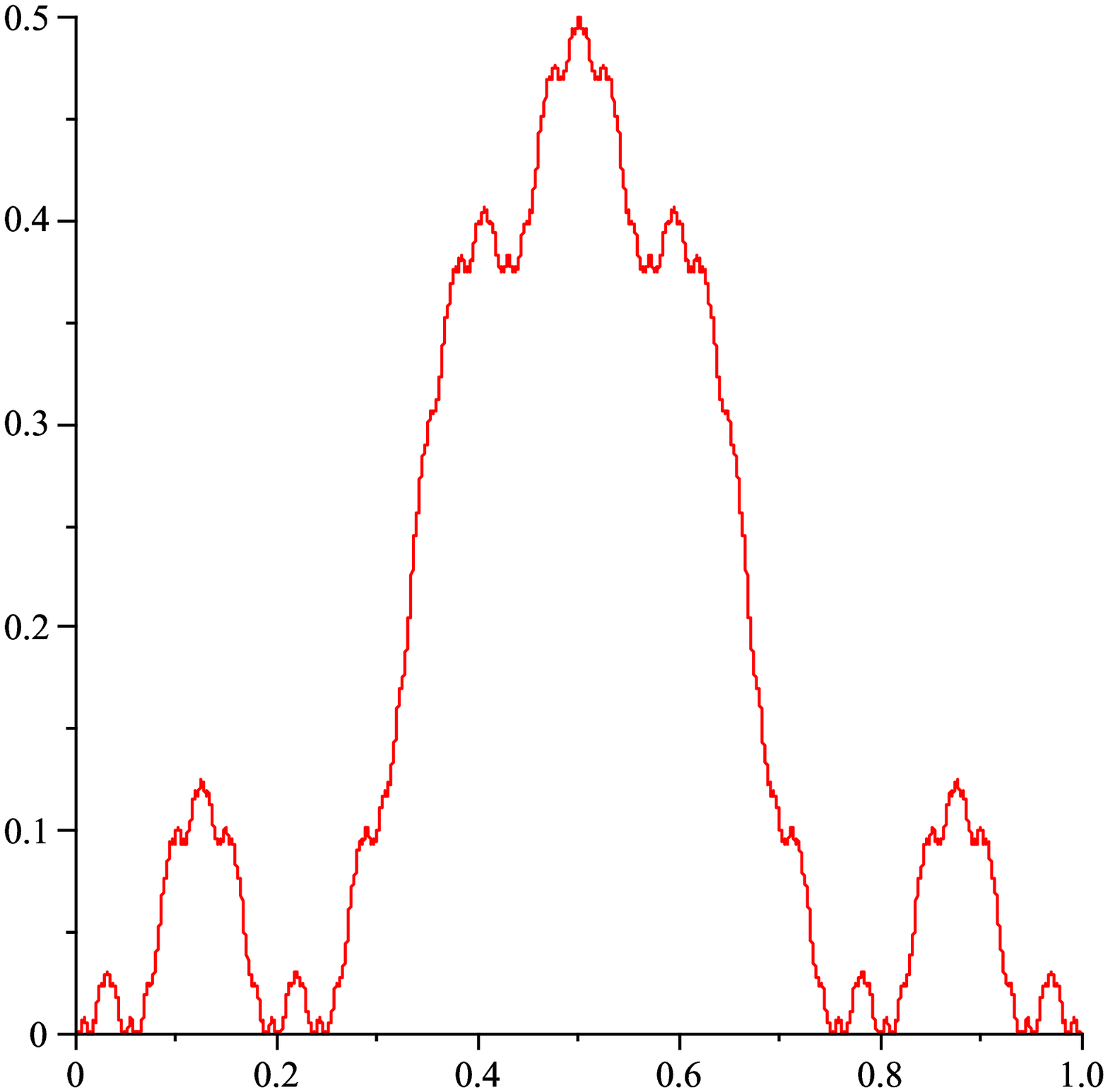, height=.25\textheight, width=.4\textwidth}
\qquad
\epsfig{file=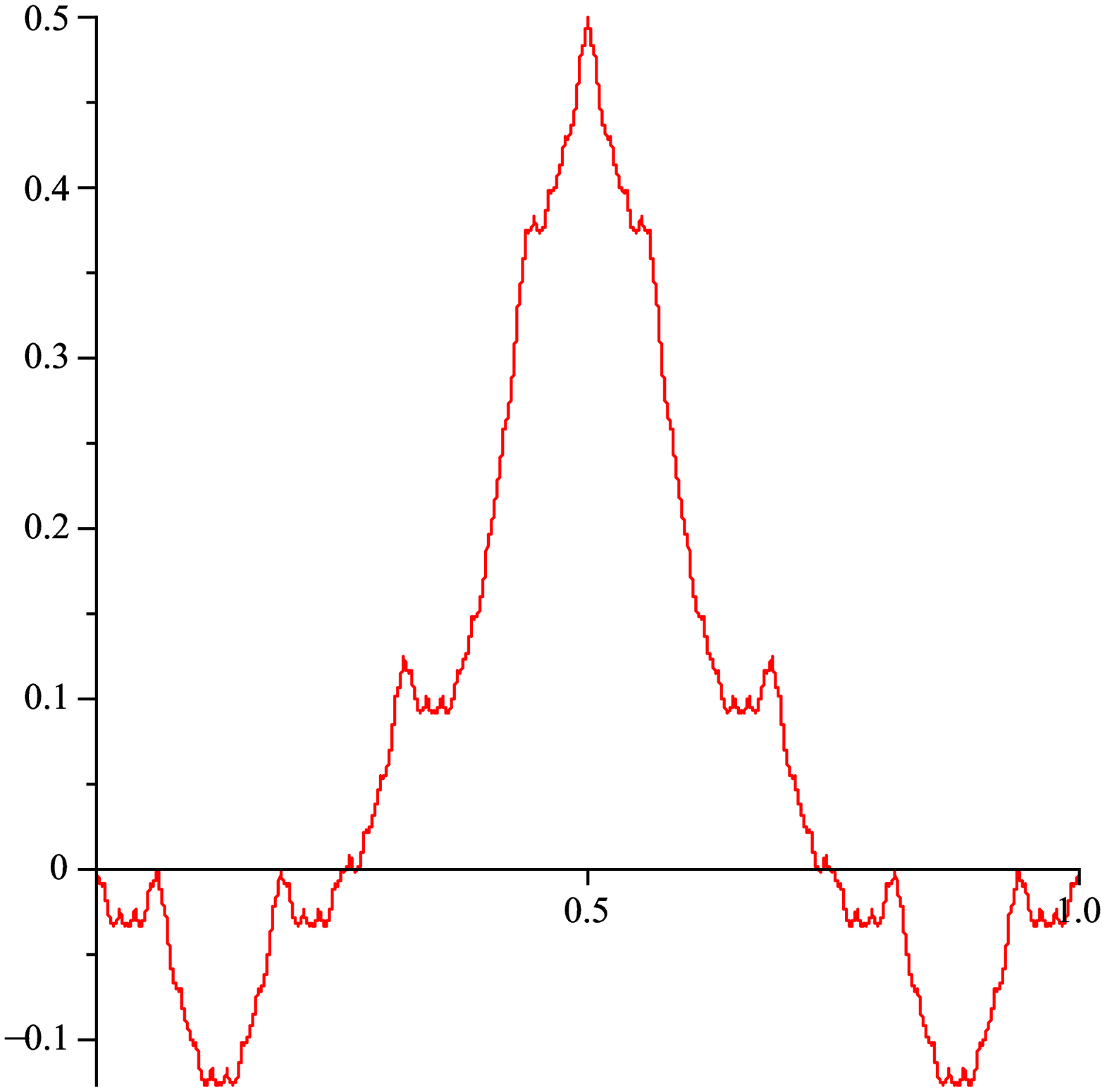, height=.25\textheight, width=.4\textwidth}
\caption{Two functions of the form \eqref{eq:general-function}: the alternating Takagi function (left) has $r_n=(-1)^n$; the function at right has $r_n=1$ if $n\equiv 0\ (\mod 3)$, $r_n=-1$ otherwise.}
\label{fig:general-examples}
\end{center}
\end{figure}

The first natural question to ask is, for which values of $y$ the level set $L_f(y)$ is nonempty. Since $f$ is continuous, its range is a compact interval which we denote by $[m(f),M(f)]$. Define the height of the graph of $f$ by
\begin{equation*}
\height(f):=\max_{x\in[0,1]}f(x)-\min_{x\in[0,1]}f(x)=M(f)-m(f).
\end{equation*}

\begin{theorem} \label{prop:height}
Put $s_n:=r_0+r_1+\dots+r_{n-1}$ for $n\in\NN$, and for $j\in\ZZ \backslash\{0\}$, let $\tau_j:=\inf\{n:s_n=j\}$. Then 
\begin{equation*}
M(f)=\sum_{k=1}^\infty \left(\frac12\right)^{\tau_{2k-1}} \qquad\mbox{and} \qquad  m(f)=-\sum_{k=1}^\infty \left(\frac12\right)^{\tau_{1-2k}},
%\label{eq:max-and-min}
\end{equation*}
where ${(1/2)}^\infty$ is interpreted as zero. Consequently, 
\begin{equation}
\height(f)=\sum_{j\in\ZZ} \left(\frac12\right)^{\tau_{2j-1}},
\label{eq:graph-height}
\end{equation}
and in particular, $1/2\leq \height(f)\leq 2/3$. There are uncountably many  $f\in\TT_{\pm}$ with $\height(f)=1/2$, but there are only four functions $f\in\TT_{\pm}$ with $\height(f)=2/3$: the functions $\pm\,T$ and the functions $\pm\,\tilde{T}$, where $\tilde{T}(x)=T(x)-2\phi(x)$.
%\begin{equation*}
%\tilde{T}(x)=-\phi(x)+\sum_{n=1}^\infty \frac{1}{2^n}\phi(2^n x).
%\end{equation*}
\end{theorem}

\begin{proof}
Since $f(0)=0$ and $f(1/2)=\pm\,1/2$, we have $\height(f)\geq 1/2$. 
The expression for $M(f)$ was derived in \cite[Theorem 1.1]{Allaart1}.
(Note that in \cite{Allaart1}, $\phi$ is defined as $2\dist(x,\ZZ)$ and hence the expression given there has an extra factor $2$.) The expression for $m(f)$ follows analogously, by symmetry. Subtracting the two expressions gives \eqref{eq:graph-height}. 
Since the $\tau_{2j-1}$ are all distinct and odd, it follows that 
\begin{equation*}
\height(f)\leq\sum_{n=1}^\infty {\left(\frac12\right)}^{2n-1}=\frac23.
\end{equation*}
Observe next that $\height(f)=1/2$ if and only if $\tau_{2j-1}=\infty$ for all $j$ but one. (Clearly, either $\tau_1$ or $\tau_{-1}$ is equal to $1$.) This is the case if and only if $\{s_n\}$ remains either in $\{0,1,2\}$ or in $\{0,-1,-2\}$, and there are uncountably many sequences ${\bf r}$ which satisfy this. On the other hand, $\height(f)=2/3$ if and only if every odd time is a record time for the sequence $\{|s_n|\}$. This is the case if and only if either $r_n=1$ for all $n\geq 1$, or $r_n=-1$ for all $n\geq 1$. This leaves the four cases indicated.
\end{proof}

We might remark that all four functions with height $2/3$ are identical except for translation and reflection in the $x$-axis, provided we regard them as periodic functions on $\RR$. The reader is encouraged to sketch the graphs!

Let $\lambda$ denote Lebesgue measure on $\RR$, and for any set $A$, let $|A|$ denote the number of elements in $A$. The next main result of this paper is

\begin{theorem} \label{thm:finite-ae}
For $\lambda$-almost every $y$, $L_f(y)$ is a finite set.
\end{theorem}

However, the expected number of points in $L_f(y)$ for $y$ chosen at random from the range of $f$ is infinite. This follows from a theorem of Banach (see Saks \cite[Theorem IX.6.4]{Saks}), since $f$ is nowhere differentiable and hence of unbounded variation, so that
\begin{equation*}
\int_\RR |L_f(y)|\,dy=\infty.
\end{equation*}

While Theorem \ref{thm:finite-ae} is of a probabilistic nature, the Baire category view gives a rather different picture.

\begin{theorem} \label{thm:uncountables}
The set 
\begin{equation*}
Y_f^p:=\{y\in\RR: L_f(y)\ \mbox{is a perfect set (hence uncountable)}\}
\end{equation*}
is residual in the range of $f$. 
\end{theorem}

Note that the sets $L_f(y)$ are closed and bounded. Thus, Theorems \ref{thm:finite-ae} and \ref{thm:uncountables} can be stated equivalently by saying that for Lebesgue-almost every $y$, $L_f(y)$ consists {\em entirely} of isolated points, but for a set of $y$-values residual in the range of $f$, $L_f(y)$ has {\em no} isolated points. This underscores the contrast between the two theorems.

The last main result of this paper concerns local level sets. The definition of a local level set given by Lagarias and Maddock \cite{LagMad1} can be written in a way that makes sense for all $f\in\TT_{\pm}$. We show that the average number of local level sets contained in a level set of $f$ is equal to the reciprocal of the height of $f$, and conclude by Theorem \ref{prop:height} that this average is at least $3/2$, but at most $2$. A precise statement of this result involves some notation and preliminary lemmas, and will be deferred until Section \ref{sec:local}.

The rest of this paper is organized as follows. Section \ref{sec:prelim} collects several elementary properties of functions in $\TT_{\pm}$, and introduces important notation and useful concepts. Section \ref{sec:key-proposition} states and proves a key proposition on which the main results of this paper depend. In Section \ref{sec:proofs} we prove Theorems \ref{thm:finite-ae} and \ref{thm:uncountables}. Finally, in Section \ref{sec:local}, we give a precise definition of local level sets and state and prove the theorem about the average number of local level sets contained in a level set of $f\in\TT_{\pm}$.

\section{Preliminaries} \label{sec:prelim}

Throughout this paper, $\ZZ_+$ denotes the set of nonnegative integers, $\NN$ the set of positive integers, and $\pi_Y(A)$ the projection of a set $A\subset \RR^2$ onto the $y$-axis.

In this section, consider $f\in\TT_{\pm}$ arbitrary but fixed.
An important aspect of $f$ is its symmetry with respect to $x=1/2$:
\begin{equation*}
f(1-x)=f(x), \qquad x\in[0,1].
\end{equation*}
Define the partial sums
\begin{equation*}
f_k(x):=\sum_{n=0}^{k-1}\frac{r_n}{2^n}\phi(2^n x), \qquad k\in\NN.
\end{equation*}
Each $f_k$ is piecewise linear with integer slopes. In fact, the slope of $f_k$ at a non-dyadic point $x$ is easily expressed in terms of the binary expansion of $x$. We write $x\in[0,1)$ in binary as
\begin{equation*}
x=\sum_{n=1}^\infty \frac{\eps_n}{2^n}=0.\eps_1\eps_2\dots\eps_n\dots, \qquad\eps_n=\eps_n(x)\in\{0,1\},
%\label{eq:binary-expansion}
\end{equation*}
with the convention that if $x$ is dyadic rational, we take the representation ending in all zeros. For $k\in\ZZ_+$, let
\begin{equation*}
D_k(x):=\sum_{j=1}^k r_{j-1}(-1)^{\eps_j}.
\end{equation*}
It follows directly from \eqref{eq:general-function} that the slope of $f_k$ at a non-dyadic point $x$ is $D_k(x)$. 

The following formula for $f(x)$ generalizes an expression for the Takagi function given by Lagarias and Maddock \cite[p.~212]{LagMad1}.

\begin{lemma} \label{lem:general-D-expression}
The function $f$ can be expressed in terms of ${\bf r}$ and $\{D_n(x)\}$ by
\begin{equation}
f(x)=C({\bf r}) -\frac14\sum_{n=1}^\infty(-1)^{\eps_{n+1}}\frac{D_n(x)}{2^n},
\label{eq:general-D-expression}
\end{equation}
where $C({\bf r})=\sum_{n=0}^\infty r_n/2^{n+2}$ is independent of $x$.
\end{lemma}

\begin{proof}
Let $\omega_k=(-1)^{\eps_k}$. Then we can write $\phi(x)$ as
\begin{equation*}
\phi(x)=\sum_{k=1}^\infty \frac{\eps_k(1-\eps_1)+(1-\eps_k)\eps_1}{2^k} =\sum_{k=2}^\infty \frac{1-\omega_1 \omega_k}{2^{k+1}}.
\end{equation*}
Similarly,
\begin{equation*}
\phi(2^n x)=\sum_{k=2}^\infty \frac{1-\omega_{n+1}\omega_{n+k}}{2^{k+1}}.
\end{equation*}
Substituting this into \eqref{eq:general-function} we obtain
\begin{align*}
f(x)&=\sum_{n=0}^\infty \sum_{k=2}^\infty \frac{r_n}{2^{n+k+1}}(1-\omega_{n+1} \omega_{n+k})
=\sum_{n=0}^\infty \sum_{j=n+2}^\infty \frac{r_n}{2^{j+1}}(1-\omega_{n+1}\omega_j)\\
&=C({\bf r})-\sum_{j=2}^\infty \frac{\omega_j}{2^{j+1}} \sum_{n=0}^{j-2} r_n \omega_{n+1}
=C({\bf r})-\sum_{j=2}^\infty \frac{(-1)^{\eps_j}}{2^{j+1}} D_{j-1}(x),
\end{align*}
and re-indexing gives \eqref{eq:general-D-expression}.
\end{proof}

The formula \eqref{eq:general-D-expression} yields an easy proof of the following important fact. The proof of this lemma, as well as the proofs of Lemmas \ref{lem:hump-count} and \ref{lem:leading-humps} below, were already given in \cite{Allaart3}. But since the proofs are short, they are repeated here (with minor notational changes) for completeness.

\begin{lemma} \label{lem:equivalence}
If $|D_n(x)|=|D_n(x')|$ for every $n$, then $f(x)=f(x')$.
\end{lemma}

\begin{proof}
Let $Z=\{n\geq 0: D_n(x)=0\}$, and enumerate the elements of $Z$ as $0=n_0<n_1<n_2<\dots$. If $D_n(x)=-D_n(x')$ for $n_i<n<n_{i+1}$ (where possibly $n_{i+1}=\infty$), then it must be the case that $\eps_n(x)=1-\eps_n(x')$ for $n_i<n\leq n_{i+1}$. Since $D_{n_{i+1}}(x)=D_{n_{i+1}}(x')=0$, it follows that the part of the summation in \eqref{eq:general-D-expression} over $n_i<n\leq n_{i+1}$ is the same for $x$ as for $x'$. This can be repeated for all intervals $(n_i,n_{i+1}]$ on which $D_n(x)\neq D_n(x')$. Hence, $f(x)=f(x')$.
\end{proof}

An important consequence of Lemma \ref{lem:equivalence} is that, if $D_n(x)=0$ for infinitely many indices $n$, then the level set $L_f(f(x))$ contains a Cantor set, as there are $2^{\aleph_0}$ many points $x'$ such that $|D_n(x)|=|D_n(x')|$ for each $n$.

The term `balanced' in the following definition was introduced by Lagarias and Maddock \cite{LagMad1}.

\begin{definition} \label{def:balanced}
{\rm
A dyadic rational of the form $x=0.\eps_1\eps_2\dots\eps_{2m}$ is called {\em balanced} if $D_{2m}(x)=0$. If there are exactly $n$ indices $1\leq j\leq 2m$ such that $D_j(x)=0$, we say $x$ is a balanced dyadic rational of {\em generation} $n$. By convention, we consider $x=0$ to be a balanced dyadic rational of generation $0$.

For $n\in\ZZ_+$, the set of balanced dyadic rationals of generation $n$ is denoted by $\BB_n$. We also set $\BB=\bigcup_{n=0}^\infty \BB_n$, so $\BB$ is the set of all balanced dyadic rationals.
}
\end{definition}

Next, let
\begin{equation*}
\mathcal{G}_f:=\{(x,f(x)): 0\leq x\leq 1\}
\end{equation*}
denote the graph of $f$ over the unit interval $[0,1]$.

\begin{lemma} \label{lem:similar-copies}
Let $m\in\NN$, and let $x_0=k/2^{2m}=0.\eps_1\eps_2\dots\eps_{2m}$ be a balanced dyadic rational. Then for $x\in[k/2^{2m},(k+1)/2^{2m}]$, we have
\begin{equation*}
f(x)=f(x_0)+2^{-2m}g\left(2^{2m}(x-x_0)\right),
\end{equation*}
for some function $g\in\TT_{\pm}$. Moreover, the function $g$ depends on $x_0$ only through $m$.
In other words, the part of $\GG_f$ above the interval $[k/2^{2m},(k+1)/2^{2m}]$ is a similar copy of $\mathcal{G}_g$ for some $g\in\TT_{\pm}$, reduced by a factor $2^{2m}$ and shifted up by $f(x_0)$.
\end{lemma}

\begin{proof}
This is immediate from the definition \eqref{eq:general-function}, since the slope of $f_{2m}$ over the interval $[k/2^{2m},(k+1)/2^{2m}]$ is equal to $D_{2m}(x_0)=0$, and $f(x_0)=f_{2m}(x_0)$. In fact, $g=f_{{\bf r}'}$, where ${\bf r}' =(r_{2m},r_{2m+1},\dots)$.
\end{proof}

\begin{definition} \label{def:humps}
{\rm
For a balanced dyadic rational $x_0=k/2^{2m}$, define 
\begin{equation*}
I(x_0):=[k/2^{2m},(k+1)/2^{2m}],
\end{equation*}
and
\begin{equation*}
H(x_0):=\{(x,f(x)): x\in I(x_0)\}.
\end{equation*}
By Lemma \ref{lem:similar-copies}, $H(x_0)$ is a similar copy of $\mathcal{G}_g$ for some $g\in\TT_{\pm}$; we call $H(x_0)$ a {\em hump} and $m$ its {\em order}. The {\em generation} of $H(x_0)$ is the generation of the balanced dyadic rational $x_0$, and a hump of generation $1$ will also be called a {\em first-generation hump}. By convention, $\mathcal{G}_f$ itself is a hump of generation $0$. If $D_j(x_0)\geq 0$ for every $j\leq 2m$, we call $H(x_0)$ a {\em principal hump}. We denote the set of all principal humps by $\HH_p$.
}
\end{definition}

Each hump of generation $n$ is contained in a hump of generation $n-1$. Different humps may have different shapes, but all are left-to-right symmetric, and any two humps of the same order $m$ are identical copies of each other. We will need to count humps and principal humps of a given order. To this end, we introduce the Catalan numbers
\begin{equation*}
C_n:=\frac{1}{n+1}\binom{2n}{n}, \qquad n\in\ZZ_+.
\end{equation*}
It is well known that
\begin{equation}
\sum_{n=0}^\infty C_n\left(\frac14\right)^n=2.
\label{eq:Catalan-gf}
\end{equation}

\begin{lemma} \label{lem:hump-count} 
Let $m\in\NN$.
\begin{enumerate}[(i)]
\item There are $\binom{2m}{m}$ humps of order $m$.
\item There are $C_m$ principal humps of order $m$.
\end{enumerate}
\end{lemma}

\begin{proof}
Each hump of order $m$ corresponds uniquely to a path of $m$ steps starting at $(0,0)$, taking steps $(1,1)$ or $(1,-1)$, and ending at $(2m,0)$. There are $\binom{2m}{m}$ such paths, proving (i). It is well known that exactly $C_m$ of these paths stay on or above the horizontal axis (see Feller \cite[p.~73]{Feller}), which gives (ii). 
\end{proof}

The projections of the first-generation humps onto the $x$-axis have disjoint interiors, and the union of these projections is dense in $[0,1]$. In fact, we have more:

\begin{lemma} \label{lem:dense-x-union}
The union
\begin{equation*}
U_n:=\bigcup_{x_0\in\BB_n}I(x_0)
\end{equation*}
is dense in $[0,1]$ for each $n\in\NN$.
\end{lemma}

\begin{proof}
Considered as a stochastic process on $[0,1]$ with Lebesgue measure, $\{D_n\}$ is a symmetric simple random walk, and therefore returns to $0$ infinitely many times with probability one. This implies that for each $n$ the set $U_n$ has full measure in $[0,1]$. Hence, $U_n$ is dense.
\end{proof}

\section{The key proposition} \label{sec:key-proposition}

Recall that $\BB_n$ denotes the set of balanced dyadic rationals of generation $n$. Define a set
\begin{equation*}
X:=[0,1]\backslash \bigcup_{x_0\in\BB_1} I^\circ(x_0),
\end{equation*}
where $I^\circ(x_0)$ denotes the interior of $I(x_0)$. Observe that $X$ is closed. For $x_0\in\BB_n$, define the subset $X(x_0)$ of $I(x_0)$ similarly by
\begin{equation*}
X(x_0):=I(x_0)\backslash \bigcup_{x_1\in\BB_{n+1}} I^\circ(x_1).
\end{equation*}
We call the graph of $f$ restricted to $X(x_0)$ a {\em truncated hump}, and denote it by $H^t=H^t(x_0)$.

Since $X$ is symmetric, the restriction of $f$ to $X$ is symmetric as well. The importance of $X$ is made clear by the following proposition, which is key to proving the main results of this paper.

\begin{proposition} \label{lem:bijection}
\begin{enumerate}[(i)]
\item We have
\begin{equation*}
f(X)=\begin{cases}
[0,1/2], & \mbox{if $r_0=1$}\\
[-1/2,0], & \mbox{if $r_0=-1$}.
\end{cases}
\end{equation*}
\item If $y\not\in f(\BB_1)$, then $L_f(y)$ intersects $X\cap[0,1/2]$ in at most one point.
%\item For every $y$, $L_f(y)\cap X$ is countable.
\end{enumerate}
\end{proposition}

The proof of the proposition uses the following auxiliary functions. 
Let
\begin{equation*}
f^*(x):=\begin{cases}
f(x), & \mbox{if $x\in X$},\\
f(x_0), & \mbox{if $x\in I(x_0)$, where $x_0\in\BB_1$}.
\end{cases}
\end{equation*}
Note that $f^*$ is well defined, because if $x$ is an endpoint of an interval $I(x_0)$, we have $f(x)=f(x_0)$. Define piecewise linear approximants of $f^*$ by
\begin{equation*}
f_n^*(x):=\begin{cases}
f(x_0), & \mbox{if $x\in I(x_0)$ with $x_0=k/2^{2m}\in\BB_1$ and $2m\leq n$},\\
f_n(x), & \mbox{otherwise}.
\end{cases}
\end{equation*}
Thus, $f_n^*$ permanently ``fixes" each first-generation flat segment in the piecewise linear approximations of $f$.

\begin{lemma} \label{lem:f-star}
Suppose $r_0=1$. The functions $f^*$ and $f_n^*$ have the following properties:
\begin{enumerate}[(i)]
\item $f_n^*$ is continuous and nondecreasing on $[0,1/2]$ for every $n$;
\item $f_n^*\to f^*$ uniformly on $[0,1]$;
\item $f^*$ is continuous and nondecreasing on $[0,1/2]$.
\end{enumerate}
\end{lemma}

\begin{proof}
Statement (i) is obvious for $n=1$, as $f_1^*=f_1$. Assume the statement holds for $n\in\NN$, and suppose that $r_n=1$. In the transition from $f_n^*$ to $f_{n+1}^*$, each horizontal line segment in the graph of $f_n^*$ stays fixed, and each line segment of strictly positive integer slope $m$ is replaced with two connecting line segments of slopes $m+1$ and $m-1$, respectively, which meet the original line segment at its endpoints. The case $r_n=-1$ is similar. Thus, it follows inductively that (i) holds for every $n\in\NN$.

Statement (ii) follows from the uniform convergence of $f_n$ to $f$ by the following argument. Let $x\in[0,1]$. If $x\in X$, then $f^*(x)-f_n^*(x)=f(x)-f_n(x)$. Suppose instead that $x\in I(x_0)$, with $x_0=k/2^{2m} \in\BB_1$. If $2m\leq n$, then the flat segment in the graph of $f_{2m}$ above $I(x_0)$ has already been fixed by $f_n^*$, and so $f^*(x)-f_n^*(x)=f(x_0)-f(x_0)=0$. Finally, if $2m>n$, then 
\begin{equation*}
|f^*(x)-f_n^*(x)|=|f(x_0)-f_n(x)|\leq |f(x_0)-f(x)|+|f(x)-f_n(x)|.
\end{equation*}
Since $|x-x_0|\leq 2^{-2m}<2^{-n}$ in this case, $|f(x_0)-f(x)|$ can be made uniformly small by choosing $n$ sufficiently large, in view of the uniform continuity of $f$. Since $f_n\to f$ uniformly, an examination of the three cases above yields statement (ii).

Statement (iii), of course, is a direct consequence of (i) and (ii).
\end{proof}

\begin{proof}[Proof of Proposition \ref{lem:bijection}]
Assume throughout that $r_0=1$; the case $r_0=-1$ is entirely similar. 

We claim first that $f(X)=f^*([0,1])$. If $y\in f(X)$, then the definition of $f^*$ immediately gives $y\in f^*(X)\subset f^*([0,1])$. Assume now that $y\in f^*([0,1])$, and let $x\in[0,1]$ such that $f^*(x)=y$. If $x\in X$ we are done; otherwise, $x\in I(x_0)$ for some $x_0\in\BB_1$, in which case $f(x_0)=f^*(x)=y$. Since $x_0\in X$, the claim is proved.

Now observe that $f^*(0)=0$ and $f^*(1/2)=1/2$. Lemma \ref{lem:f-star}(iii) and the symmetry of the graph of $f$ therefore give $f^*([0,1])=[0,1/2]$, proving (i).

To prove (ii), suppose there exist $x,x'\in X\cap[0,1/2]$ with $x<x'$ such that $f(x)=f(x')=y$. By Lemma \ref{lem:dense-x-union} the interval $(x,x')$ intersects some interval $I(x_0)$ with $x_0\in\BB_1$. Since $x\in X$, $x$ can not lie in the interior of $I(x_0)$, and hence $x\leq x_0<x'$. Now let $n\geq 2m$, where $m$ is the order of $x_0$. Then
\begin{equation*}
f_n(x')=f_n^*(x')\geq f_n^*(x_0)\geq f_n^*(x)=f_n(x). 
\end{equation*}
But $f_n^*(x_0)=f_n(x_0)$, so letting $n\to\infty$ we obtain $f(x')\geq f(x_0)\geq f(x)$. But this means $f(x_0)=y$, so $y\in f(\BB_1)$. This proves (ii).
%Since $\BB_1$ is countable, statement (iii) will follow once we show that, if $x,x',x''$ are points in $L_f(y)\cap X\cap[0,1/2]$ with $x'<x<x''$, then $x$ must be an endpoint of $I(x_0)$ for some $x_0\in\BB_1$. Let $x,x',x''$ be as above. By the argument from the proof of part (ii), there are points $x_1$ and $x_2$ in $\BB_1$ such that $x'\leq x_1<x\leq x_2<x''$, and $f(x_1)=f(x_2)=y$. Then for all large enough $n$, $f_n^*=y$ everywhere on $I(x_1)\cup I(x_2)$, and since $f_n^*$ is nondecreasing on $[0,1/2]$, it follows that $f_n^*$ is constant on $[x_1,x_2]$. But then $f_k$ is constant on some interval containing $x$ for some $k\leq n$. This means $x\in I(x_0)$ for some $x_0\in\BB_1$, and being a member of $X$, $x$ must be an endpoint of $I(x_0)$. This proves the claim and part (iii) of the proposition.
\end{proof}

In what follows, let $l_y$ denote the horizontal line at level $y$. That is,
\begin{equation*}
l_y:=\{(x,y): x\in\RR\}.
\end{equation*}

\begin{corollary} \label{cor:hump-intersections}
Let $H^t$ be a truncated hump of generation $n$ and order $m$. Then:
\begin{enumerate}[(i)]
\item $\pi_Y(H^t)$ is an interval, and $\lambda(\pi_Y(H^t))=(1/2)(1/4)^m$;
\item If $y\not\in f(\BB_{n+1})$, then $|l_y\cap H^t|\leq 2$.
%\item For each $y$, the set $l_y\cap H^t$ is countable.
\end{enumerate}
\end{corollary}

\begin{proof}
The hump $H$ is a small-scale similar copy of $\GG_g$ for some $g\in\TT_{\pm}$, reduced by a factor $4^m$. And any hump of generation $n+1$ inside $H$ corresponds to a hump of generation $1$ in the graph of $g$. Hence, the corollary follows upon applying Proposition \ref{lem:bijection} to $g$ and using the symmetry of $H$.
\end{proof}

\section{Proofs of Theorems \ref{thm:finite-ae} and \ref{thm:uncountables}} \label{sec:proofs}

\begin{lemma} \label{lem:leading-humps}
\begin{enumerate}[(i)]
\item For every hump $H$ there is a principal hump $H'$ of the same order and generation as $H$, such that $\pi_Y(H)=\pi_Y(H')$.
\item For every principal hump $H'$, there are only finitely many humps $H$ such that $\pi_Y(H)=\pi_Y(H')$.
\end{enumerate}
\end{lemma}

\begin{proof}
(i) Let $H=H(x_0)$, where $x_0$ is a balanced dyadic rational of order $m$. There is a unique balanced dyadic rational $x_1$ of order $m$  such that $D_j(x_1)=|D_j(x_0)|$ for all $j\leq 2m$. By Lemma \ref{lem:equivalence}, $f(x_0)=f(x_1)$. By definition, $H':=H(x_1)$ is a principal hump of order $m$ and of the same generation as $H$. Hence $H'$ is the same size as $H$ and sits at the same height in the graph of $f$. Therefore, $H'$ is as required.

(ii) This is immediate from Lemma \ref{lem:hump-count}(i).
\end{proof}

\begin{lemma} \label{lem:when-finite}
Suppose $y\not\in f(\BB)$, and $l_y$ intersects only finitely many principal humps. Then $|L_f(y)|<\infty$.
\end{lemma}

\begin{proof}
By Lemma \ref{lem:leading-humps}, $l_y$ intersects only finitely many humps.
Suppose, by way of contradiction, that $|L_f(y)|=\infty$. Then there is a maximal $n$ such that $l_y$ intersects some hump $H$ of generation $n$ in infinitely many points. But this means $|l_y\cap H^t|=\infty$, contradicting Corollary \ref{cor:hump-intersections}(ii). Hence, $|L_f(y)|<\infty$.
\end{proof}

\begin{proof}[Proof of Theorem \ref{thm:finite-ae}]
If $H$ is a principal hump of order $m$, then
\begin{equation*}
\lambda\{y: l_y\cap H\neq\emptyset\}=\lambda(\pi_Y(H))\leq \frac23\left(\frac14\right)^m,
\end{equation*}
where the inequality follows from Theorem \ref{prop:height}, since $H$ is a similar copy of the graph of some function $g\in\TT_{\pm}$, scaled by $1/4^m$. This gives
\begin{equation*}
\sum_{H\in\HH_p} \lambda\{y: l_y\cap H\neq\emptyset\} \leq \frac23\sum_{m=0}^\infty C_m\left(\frac14\right)^m<\infty
\end{equation*}
by \eqref{eq:Catalan-gf} and Lemma \ref{lem:hump-count}. Thus, by the first Borel-Cantelli lemma,
\begin{equation*}
\lambda\{y: l_y\ \mbox{intersects infinitely many $H$ in $\HH_p$}\}=0.
\end{equation*}
Because $f(\BB)$ is a countable set, Lemma \ref{lem:when-finite} gives the desired result.
\end{proof}

%Next, we prove Theorem \ref{thm:uncountables}.

\begin{proof}[Proof of Theorem \ref{thm:uncountables}]
Each function $f$ in $\TT_{\pm}$ is monotone on no interval. This follows from the nowhere-differentiablility of $f$ mentioned in the Introduction, but it can also be seen directly as follows: If $x_0\in\BB_n$, then $\diam(I(x_0))\leq 4^{-n}$. So if $J$ is an arbitrary subinterval of $[0,1]$, we can choose $n$ large enough so that $\diam(J)>2\diam(I(x_0))$ for each $x_0\in\BB_n$. By Lemma \ref{lem:dense-x-union}, $J$ then contains at least one interval $I(x_0)$ with $x_0\in\BB_n$. Let $d:=\diam(I(x_0))$. Then $f(x_0)=f(x_0+d)$, but $f(x_0+d/2)=f(x_0)\pm d/2\neq f(x_0)$. Hence, $f$ is not monotone on $J$.

It was shown by Garg \cite[Theorem 1]{Garg} that for {\em any} continuous function $f$ which is monotone on no interval, the level set $L_f(y)$ is perfect for a set of $y$-values residual in the range of $f$. 
This applies in particular to $f$ in $\TT_{\pm}$.
\end{proof}

%We show that
%\begin{equation}
%Y_f^\infty\supset\bigcap_{n=1}^\infty \bigcup_{x_0\in\BB_n}J(x_0).
%\label{eq:infinites}
%\end{equation}
%Set
%\begin{equation*}
%E_n:=\bigcup_{x_0\in\BB_n}J(x_0), \qquad n\in\NN.
%\end{equation*}
%Let $y\in \bigcap_{n=1}^\infty E_n$. Then for each $n$ there is a balanced dyadic rational $x_n\in\BB_n$ such that $y\in J(x_n)$. Fix $n$ for the moment. We can write $x_n=k/2^{2m}$ for some $m\in\NN$ and integer $k$, and there are exactly $n$ integers $j: 1\leq j\leq 2m$ such that $D_j(x_n)=0$. But then there are $2^n$ points $\xi_1,\dots,\xi_{2^n}$ in $\BB_n$ (one of them being $x_n$) such that $|D_j(\xi_i)|=|D_j(x_n)|$ for all $j\in\NN$ and $i=1,\dots,2^n$. By Lemma \ref{lem:equivalence}, the humps $H(\xi_i)$, $i=1,\dots,2^n$ all have the same projection onto the $y$-axis. Hence, the line $l_y$ intersects at least $2^n$ distinct humps, and so $|L_f(y)|\geq 2^{n+1}$. Since $n$ was arbitrary, it follows that $|L_f(y)|=\infty$. Thus, we have \eqref{eq:infinites}.

%Since $E_n=f(U_n)$ with $U_n$ defined as in Lemma \ref{lem:dense-x-union}, it follows from that lemma and the continuity of $f$ that $E_n$ is dense in the range of $f$. Now $E_n$ is a countable union of intervals, so removing the endpoints of these intervals does not ruin denseness. In other words, the set
%\begin{equation*}
%O_n:=\bigcup_{x_0\in\BB_n}J^\circ(x_0)
%\end{equation*}
%is still dense in the range of $f$. By \eqref{eq:infinites}, $Y_f^\infty$ contains $\bigcap_{n=1}^\infty O_n$, which is a dense $G_\delta$ set by Baire's theorem. Therefore, $Y_f^\infty$ is residual in the range of $f$.

\section{Local level sets} \label{sec:local}

Lagarias and Maddock \cite{LagMad1,LagMad2} introduced the concept of a local level set of the Takagi function. Their definition makes sense also for our general setting. First, define an equivalence relation on $[0,1)$ by
\begin{equation*}
x\sim x'\quad \stackrel{\defin}{\Longleftrightarrow}\quad |D_j(x)|=|D_j(x')|\ \mbox{for each $j\in\NN$}.
%\label{eq:equivalence-relation}
\end{equation*}
The {\em local level set} containing $x$ is then defined by 
\begin{equation*}
L_x^{loc}:=\{x': x'\sim x\}.
\end{equation*}
By Lemma \ref{lem:equivalence}, each local level set is contained in some (global) level set. Note further that each local level set $L_x^{loc}$ contains a unique point $x'$ such that $D_j(x')\geq 0$ for all $j$. We call this point $x'$ the {\em principal point} of $L_x^{loc}$.

For the Takagi function, Lagarias and Maddock \cite{LagMad1} point out that each local level set is either finite or a Cantor set. This remains true in our general setting. They prove further that for the Takagi function, the average number of local level sets contained in a level set chosen at random is $3/2$. We generalize this result here to $f\in\TT_{\pm}$.
For $y\in\RR$, let $N_f^{loc}(y)\in\NN\cup\{\infty\}$ denote the number of local level sets contained in $L_f(y)$. 

\begin{lemma} \label{lem:finite-local-level-sets}
If $y\not\in f(\BB)$, then the number of finite local level sets contained in $L_f(y)$ is exactly
\begin{equation*}
|\{H\in\HH_p:y\in\pi_Y(H^t)\}|.
\end{equation*}
\end{lemma}

\begin{proof}
We show that for each $y\not\in f(\BB)$, there is a bijection between the collection of finite local level sets in $L_f(y)$ and the collection of truncated principal humps which intersect the line $l_y$. The lemma will follow from this correspondence.

Let $L_x^{loc}\subset L_f(y)$ be a finite local level set with principal point $x$. Then there exists $m\in\ZZ_+$ such that $D_{2m}(x)=0$, but $D_j(x)>0$ for all $j>2m$. (Otherwise, $L_x^{loc}$ would be a Cantor set.) Let $x_0$ be the dyadic rational $k/2^{2m}$ whose first $2m$ binary digits coincide with those of $x$. Then $x_0$ is balanced of order $m$, $H(x_0)$ is a principal hump, and $(x,y)\in H^t(x_0)$. Thus, we have a mapping from finite local level sets in $L_f(y)$ to truncated principal humps which intersect $l_y$. This mapping is bijective. It is onto, because given a truncated principal hump $H^t$ which intersects $l_y$, let $x$ be any point such that $(x,y)\in H^t\cap l_y$; then the local level set $L_x^{loc}$ is finite and gets mapped to $H^t$ (even though $x$ may not be the principal point of $L_x^{loc}$). The mapping is one-to-one, because any truncated hump $H^t$ can have only two points of intersection with $l_y$ by Corollary \ref{cor:hump-intersections}(ii), and the abscissae of both points belong to the same local level set.
\end{proof}

\begin{theorem} \label{thm:local-level-sets}
The average number of local level sets contained in a level set $L_f(y)$ with $y$ chosen at random from the range of $f$ is given by
\begin{equation*}
\bar{N}_f^{loc}:=h_f^{-1}\int_\RR N_f^{loc}(y)\,dy=h_f^{-1},
\end{equation*}
where $h_f=\height(f)$. In particular, $3/2\leq\bar{N}_f^{loc}\leq 2$.
\end{theorem}

\begin{proof}
Since $L_f(y)$ is finite for almost every $y$, we need only consider {\em finite} local level sets. Using Lemma \ref{lem:finite-local-level-sets} and the fact that $f(\BB)$ is a Lebesgue null set, it follows that the mapping $y \mapsto N_f^{loc}(y)$ is Lebesgue measurable, and by Fubini's theorem,
\begin{equation*}
\int_\RR N_f^{loc}(y)\,dy=\sum_{H\in\HH_p} \lambda(\pi_Y(H^t))
=\frac12\sum_{m=0}^\infty C_m{\left(\frac14\right)}^m=1.
\end{equation*}
Here the second equality follows by Lemma \ref{lem:hump-count}(ii) and Corollary \ref{cor:hump-intersections}(i), and the third by \eqref{eq:Catalan-gf}. The last statement of the theorem is a consequence of Theorem \ref{prop:height}.
\end{proof}

\section*{Acknowledgments}
The author is greatly indebted to the referee for pointing out Garg's paper \cite{Garg}, thereby saving Theorem \ref{thm:uncountables} (even slightly strengthening the original), after the original argument was found to be erroneous during the examination of the galley proofs.

\end{document}